\documentclass{amsart}
\usepackage{amsmath,amssymb,amsthm,fullpage,verbatim}
\usepackage{tikz}
\newtheorem{lemma}{Lemma}

\theoremstyle{definition}
\newtheorem{definition}{Definition}
\usepackage{mathrsfs}
\begin{document}
\title{Set Equality in Combinatorial Game Theory}
\author{Michael~J.~J.~Barry}
\address{15 River Street Unit 205\\
Boston, MA 02108}
\email{mbarry@allegheny.edu}

\subjclass[2010]{91-01, 91A46}

\begin{abstract}
In Combinatorial Game Theory, the fundamental relation of game equivalence, denoted by $=$, is introduced early on and overrides the notion of set equality.  We explore what happens if set equality is given its due before game equivalence is introduced.
\end{abstract}

\maketitle

\section{Introduction}

In this note, we argue for the benefits of highlighting set equality in the axiomatic approach to Combinatorial Game Theory, particularly in regard to the sum of games and the negative of a game.

\section{Abstract Games}
Throughout this paper $=$ denotes set equality not game equivalence.

Define the ordered pair $(\emptyset,\emptyset)$ by $0$.   In the next definition, we follow Siegel~\cite[Definition 1.1]{Siegel} but use ordered pairs instead of the usual $\{ \mid \}$ notation.
\begin{definition}
Let $\tilde{\mathbb{G}}_0=\{0\}$, and for $n \geq 0$ let
\[\tilde{\mathbb{G}}_{n+1}=\{(\mathscr{G}^L , \mathscr{G}^R)\mid \mathscr{G}^L, \mathscr{G}^R \subset \tilde{\mathbb{G}}_{n}\}.\]
Then a \textbf{short game} is an element of
\[\tilde{\mathbb{G}}=\bigcup_{n \geq 0} \tilde{\mathbb{G}}_{n}.\]
\end{definition}

Note that $\tilde{\mathbb{G}}_{n}$ is a finite set for each $n \geq0$.

Thus a short game $G$ is an ordered pair $(\mathscr{G}^L, \mathscr{G}^R)$, where $\mathscr{G}^L$ and $\mathscr{G}^R$ are sets of games.  Games in  $\mathscr{G}^L$ are called the \textbf{left options} of $G$ while games in  $\mathscr{G}^R$ are called the \textbf{right options} of $G$.  The game $0$ defined above is the  game with no options.  

\begin{definition}
\cite[Definition 1.27]{Siegel} Let $G$ be a short game.  The \textbf{formal birthday} of $G$, denoted by $\tilde{b}(G)$, is the least $n$ such that $G \in \tilde{\mathbb{G}}_{n}$.
\end{definition}

Obviously, $\tilde{b}(G)=0$ if and only if $G=0$.  And if $\tilde{b}(G) =n>0$, then every option $G'$ satisfies $\tilde{b}(G')\leq n-1$, and there is at least one option $G''$ with $\tilde{b}(G'')=n-1$.  Hence if $G \neq 0$, $\tilde{b}(G)=1+\max\{\tilde{b}(G') \mid G' \in \mathscr{G}^L \cup \mathscr{G}^R\}$~\cite[Exercise 1.6]{Siegel}.

There are four games in $\tilde{\mathbb{G}}_{1}$: $0=(\emptyset,\emptyset)$, which already in $\tilde{\mathbb{G}}_{0}$, $(\{0\}, \emptyset)$, $(\emptyset, \{0\})$ and $(\{0\},\{0\}\}$.  We usually denote the second of these by $1$, the third by $-1$, and the last by $*$, see~\cite[p. 59]{Siegel}.  

So $\tilde{b}(0)=0$, and $\tilde{b}(1)=1=\tilde{b}(-1)=\tilde{b}(*)$.

\begin{definition}
Let $G=(\mathscr{G}^L, \mathscr{G}^R)$ be a short game.  The \textbf{negative} $-G$ is defined recursively by  $-G=(-\mathscr{G}^R,-\mathscr{G}^L)$,
where for a set $\mathscr{S}$ of games, 
\[-\mathscr{S}=\begin{cases}
\emptyset, & \text{if $\mathscr{S}=\emptyset$;}\\
\{-G \mid G \in \mathscr{S}\}, & \text {if $\mathscr{S}\neq \emptyset$.}
\end{cases}
\]
\end{definition}
This is a slightly more explicit version of the usual definition~\cite[Definition 1.4]{Siegel}.

So $-0=-(\emptyset, \emptyset)=(-\emptyset,-\emptyset)=(\emptyset,\emptyset)=0$
and $-1=-(0,\emptyset)=(-\emptyset, -0)=(\emptyset,  0)$, which agrees with our notation above.

\begin{lemma}
Let $G$ be a short game. Then $-(-G)=G$.
\end{lemma}

Note that here we are asserting that $-(-G)$ and $G$ are equal as sets.  Some texts and on-line course notes assert only that $-(-G) \cong G$.

\begin{proof}
By induction on $\tilde{b}(G)$.  If $\tilde{b}(G)=0$, then $G=0$ and $-(-0)=-0=0$.  When $\tilde{b}(G)>0$,
\begin{align*}
-(-G)&=-(-\mathscr{G}^R,-\mathscr{G}^L)\\
&=(-(-\mathscr{G}^L),-(-\mathscr{G}^R)\\
&=(\mathscr{G}^L,\mathscr{G}^R), && \text{by induction,}\\
&=G.
\end{align*}
\end{proof}

\begin{definition}\label{DefSumv2}
Let $G=(\mathscr{G}^L,\mathscr{G}^R)$ and $H=(\mathscr{H}^L,\mathscr{H}^R)$ be  short games.  Define the game $G+H$ recursively by
\[G+H=
((\mathscr{G}^L+H) \cup (G+\mathscr{H}^L),(\mathscr{G}^R+H) \cup (G+\mathscr{H}^R) )
\]
where for any set $\mathscr{S}$ of games and any game $K$, 
\[\mathscr{S}+K=
\begin{cases} 
\emptyset, & \text{if $\mathscr{S}=\emptyset$;}\\
\{S+K \mid S \in \mathscr{S}\}, &\text{if  $\mathscr{S} \neq \emptyset$,}
\end{cases}
\]
with a similar definition for $K+ \mathscr{S}$.
\end{definition}
Again, this is a slightly more explicit version of the usual definition~\cite[Definition 1.2]{Siegel}.

Is $G+H \in \tilde{\mathbb{G}}$?  More particularly, is $G+0$, which will arise at some stage in our calculation of $G+H$, an element of $\tilde{\mathbb{G}}$?  It would be if $G+0$ equals $G$ as sets, but this is not immediate from our definition of game sum.

Note that
\[0+0=(\emptyset,\emptyset)+(\emptyset,\emptyset)
=((\emptyset+0)\cup (0+\emptyset),(\emptyset+0)\cup (0+\emptyset))
=(\emptyset \cup\emptyset,\emptyset \cup \emptyset)=(\emptyset,\emptyset)=0.
\]

Now
\begin{align*}
1+0&=(\{0\},\emptyset)+(\emptyset,\emptyset)\\
&=(\{0+0\}\cup(1+\emptyset),(\emptyset+0) \cup(1+\emptyset))\\
&=(\{0\}\cup \emptyset,\emptyset \cup \emptyset)\\
&=(\{0\},\emptyset)\\
&=1
\end{align*}
and
\begin{align*}
1+1&=(\{0\},\emptyset)+(\{0\},\emptyset)\\
&=(\{0+1,1+0 \},(\emptyset+1) \cup (1+\emptyset))\\
&=(\{1,1 \}, \emptyset \cup \emptyset)\\
&=(\{1\}, \emptyset).
\end{align*}

\begin{lemma}\label{G+H}
Let $G \in \tilde{\mathbb{G}}_n$ and  $H \in \tilde{\mathbb{G}}_m$.  Then $G+H \in \tilde{\mathbb{G}}_{n+m}$.
\end{lemma}

\begin{proof}
By induction on $\tilde{b}(G)+\tilde{b}(H)$.  If $\tilde{b}(G)+\tilde{b}(H)=0$, then $n=0=m$, and so $G=0=H$, and $G+H=0\in \tilde{\mathbb{G}}_{0}=\tilde{\mathbb{G}}_{n+m}$.
If $\tilde{b}(G)+\tilde{b}(H)>0$, then every left option of $G+H$  has the form $G^L+H$ or $G+H^L$ for some left option $G^L$ of $G$ or $H^L$ of $G$.  By induction every left option of $G+H$ is in $\tilde{\mathbb{G}}_{n+m-1}$. Similarly every right option of $G+H$ is in  $\tilde{\mathbb{G}}_{n+m-1}$.  Hence $G+H \in \tilde{\mathbb{G}}_{n+m}$.
\end{proof}

\begin{lemma}
Let $G$ and $H$ be short games.  Then $\tilde{b}(G+H)=\tilde{b}(G)+\tilde{b}(H)$. \cite[Exercise 1.7]{Siegel}
\end{lemma}
\begin{proof}
Note that by Lemma~\ref{G+H}, $\tilde{b}(G+H)\leq \tilde{b}(G)+\tilde{b}(H)$.  By induction on $\tilde{b}(G)+\tilde{b}(H)$.  If $\tilde{b}(G)+\tilde{b}(H)=0$, then $\tilde{b}(G)=0=\tilde{b}(H)$, so $G=0=H$.  Since $0+0=0$, $\tilde{b}(0+0)=0=\tilde{b}(0)+\tilde{b}(0)$.  Assume that $\tilde{b}(G)+\tilde{b}(H)>0$.  WLOG we can assume that $\tilde{b}(G)>0$.  Then $G$ has an option $G'$, left or right, with $\tilde{b}(G')=\tilde{b}(G)-1$.  Then $G'+H$ is an option of $G+H$, and by induction $\tilde{b}(G'+H)=\tilde{b}(G')+\tilde{b}(H)=\tilde{b}(G)-1+\tilde{b}(H)$.  Hence $\tilde{b}(G+H)=\tilde{b}(G)+\tilde{b}(H)$.
\end{proof}

\begin{lemma}
Let $G$ be a short game.  Then $0+G=G=G+0$ as sets.
\end{lemma}
\begin{proof}
We will prove only that $0+G=G$.
By induction on $\tilde{b}(G)$.  If $\tilde{b}(G)=0$, then $G=0$ and $0+G=0+0=0=G$.  Assume that $\tilde{b}(G)>0$.  Then every left option of $0+G$ has the form $0+G^L$ where $G^L$ is a left option of $G$.  By induction $0+G^L=G^L$.  Thus the set of left options of $0+G$ equals the set of left options of $G$.  The corresponding argument for right options works.  It follows that $0+G=G$ as sets.
\end{proof}

Now
\[1+*=(0,\emptyset)+(0,0)=(\{1+0,0+*\},\{1+0\})=(\{1,*\},\{1\}),
\]
and

\[
*+1=(0,0)+(0,\emptyset)=(\{0+1,*+0\},\{0+1\})=(\{1,*\},\{1\})=1+*.
\]
Thus $1+*$ and $*+1$ are equal as sets.

\begin{lemma}
Let $G$ and $H$ be short games.  Then $G+H=H+G$, that is, $G+H$ and $H+G$ are equal as sets.
\end{lemma}

\begin{proof}
By induction on $\tilde{b}(G)+\tilde{b}(H)$.   f $\tilde{b}(G)+\tilde{b}(H)=0$, then $\tilde{b}(G)=0=\tilde{b}(H)$, whence $G=0=H$, and $0+0=0=0+0$.  Assume that $\tilde{b}(G+H)>0$.  Now every left option of $G+H$ has the form $G^L+H$ or $G+H^L$ for some left option $G^L$ of $G$ and or some left option $H^L$ of $H$.  By induction, $G^L+H=H+G^L$ and $G+H^L=H^L+G$.  Thus every left option of $G+L$ is a left option of $H+G$.  Similarly every left option of $H+G$ is a left option of $G+H$.  Hence the set of left options of $G+H$ equals the set of left options of $H+G$.  The argument showing equality of sets of right options is similar.  Thus $G+H=H+G$ as sets.
\end{proof}

Let us compare $(1+*)+(-1)$ and $1+(*+(-1))$.  First
\begin{align*}
(1+*)+(-1)&=(\{1,*\},\{1\})+(\emptyset,0)\\
&=(\{1+(-1),*+(-1)\},\{1+(-1),(1+*)+0\})\\
&=(\{1+(-1),*+(-1)\},\{1+(-1),1+*\}).
\end{align*}
Then since 
\[*+(-1)=(0,0)+(\emptyset,0)=(\{0+(-1)\},\{0+(-1),*+0\})=(\{-1\},\{-1,*\}),
\]
\begin{align*}
1+(*+(-1))&=(0,\emptyset)+(\{-1\},\{-1,*\})\\
&=(\{0+(*+(-1)), 1+(-1)\} ,\{ 1+(-1), 1+*\})\\
&=(\{*+(-1), 1+(-1)\},\{ 1+(-1),1+*\}).
\end{align*}

Thus $(1+*)+(-1)$ and $1+(*+(-1))$ are equal as sets.

\begin{lemma}
Let $G$, $H$, and $K$ be short games.  Then $(G+H)+K$ and $G+(H+K)$ are equal as sets.
\end{lemma}

\begin{proof}
By induction on $\tilde{b}(G)+\tilde{b}(H)+\tilde{b}(K)$.   If $\tilde{b}(G)+\tilde{b}(H)+\tilde{b}(K)=0$, then $\tilde{b}(G)=0=\tilde{b}(H)=\tilde{b}(K)$, so $G=0=H=K$, and $(0+0)+0=0+0=0+(0+0)$.  Assume that $\tilde{b}(G)+\tilde{b}(H)+\tilde{b}(K)>0$.
A typical left option of $(G+H)+K$ has the form $(G^L+H)+K$, $(G+H^L)+K$, or $(G+H)+K^L$ where $G^L$, $H^L$, and $K^L$ are left options of $G$, $H$, and $K$ respectively.  By induction,
$(G^L+H)+K=G^L+(H+K)$, $(G+H^L)+K=G+(H^L+K)$, and $(G+H)+K^L=G+(H+K^L)$ as sets.  But $G^L+(H+K)$, $G+(H^L+K)$, and $G+(H+K^L)$ are left options of $G+(H+K)$.  Thus every let option of $(G+H)+K$ is a left option of $G+(H+K)$.  Similarly every left option of $G+(H+K)$ is a left option of $(G+H)+K$.  Hence the set of left options of $(G+H)+K$ equals the set of left options of $G+(H+K)$. The argument showing equality of sets of right options is similar. It follows that $(G+H)+K$ and $G+(H+K)$ are equal as sets.
\end{proof}

\section{Conclusion}

We have shown that from the above definitions of the negative of a game and the sum of games, we can prove set equality and not just game equivalence for the pairs of games $-(-G)$ and $G$, $G+H$ and $H+G$, and $(G+H)+K$ and $G+(H+K)$.

\end{document}